\documentclass{mcrfOF}
\usepackage{amsmath}
  \usepackage{paralist}
  \usepackage{tikz}
  \usepackage{graphics} 
  \usepackage{epsfig} 
\usepackage{graphicx}  \usepackage{epstopdf}
 \usepackage[colorlinks=true]{hyperref}
\hypersetup{urlcolor=blue, citecolor=red}

\def\currentvolume{X}
 \def\currentissue{X}
  \def\currentyear{200X}
   \def\currentmonth{XX}
    \def\ppages{X--XX}
     \def\DOI{10.3934/mcrf.2021024}

\newtheorem{theorem}{Theorem}[section]

\newtheorem{lemma}[theorem]{Lemma}

\theoremstyle{definition}
\newtheorem{definition}[theorem]{Definition}
\newtheorem{remark}{Remark}

\title[Constrained control for quasilinear parabolic PDEs] 
      {Controllability under positive constraints for quasilinear parabolic PDEs}

\author[Miguel R. Nu{\~n}ez-Ch\'avez]{}

\subjclass{Primary: 35K59, 93C20; Secondary: 35K15.}
 \keywords{Quasilinear parabolic PDE, positive constraints, local controllability, stair-case method, positive minimal time.}

 \email{miguelnunez@opendeusto.es, miguelnunez@id.uff.br}

\thanks{$^*$ Corresponding author: Miguel R. Nu\~nez-Ch\'avez}

\begin{document}
\maketitle

\centerline{\scshape Miguel R. Nu\~nez-Ch\'avez$^*$}
\medskip
{\footnotesize
 \centerline{DeustoTech, Fundaci\'on Deusto}
   \centerline{ Av. Universidades, 24, 48007, Bilbao, Basque Country, Spain}
   \centerline{ Instituto de Matem\'atica e Estat\'istica, Universidade Federal Fluminense}
   \centerline{ R. Prof. Marcos Waldemar de Freitas, s/n, 24210-201, Niter\'oi, RJ, Brazil}
} 



\bigskip

 \centerline{(Communicated by Luz de Teresa)}

\begin{abstract}
This paper deals with the analysis of the internal  controllability with constraint of positive kind of a quasilinear parabolic PDE.
	We prove two results about this PDE: First, we prove a global steady state constrained controllability result. For this purpose, we employ the called ``stair-case method''. And second, we prove a global trajectory constrained controllability result. For this purpose, we employ the well-known ``stabilization property'' in $L^2$ norms. Furthermore, for both results an important argument is needed: the exact local controllability to trajectories. Then we prove the positivity of the minimal controllability time using arguments of comparison principle.  Some additional comments and open problems concerning other systems are presented.
\end{abstract}

\section{Introduction}

In this paper we focus on the controllability problem for quasilinear heat equations under constraints. Our aim is to analyse if the parabolic equation under consideration can be driven to a desired final target by means of the control action, but preserving some constraints on the control. We focus on nonnegativity constraints.

The class of linear, semilinear and quasilinear parabolic systems, in the absence of constraints, are controllable in any positive time (see \cite{ATF, Im, Mig, FZ, FI, ImYa, LEB, Lions, XU, Schm}). Sometimes, controls achieving the target at the final time are restrictions of solutions of the adjoint system. These controls experience large oscillations in the proximity of the final time. In particular, when the time horizon is too short, these oscillations prevent the control to fulfill the positivity constraint.

Now about controllability under constraints in parabolic systems, you have the first work in \cite{LOH} (2017), the authors dealt with a linear heat equation in dimension N, considering various types of boundary value problem, they also do several numerical simulations with interesting results, then in \cite{PZ} (2018) the authors worked with a semilinear equation with $C^1$ nonlinearity, without sign or globally Lipschitz assumptions on the nonlinear term. In \cite{POU} (2019) the authors worked with reaction-diffusion equations about the same question. In \cite{LEBAL} (2019) the author worked with the controllability of a $4 \times 4$ quadratic reaction-diffusion system.

In the present paper, inspired by \cite{LOH} and \cite{PZ}, we prove a more general result, this is, a problem with nonlinearity in the diffusion term. First, for steady states, the method of proof uses a ``stair-case argument", that consists in moving from one steady state to a neighbouring one, using small amplitude controls, in a recursive manner, so to reach the final target after a number of iterations and preserving the constraints on the control imposes a priori. This stair-case method, though, leads to constrained control results only when the time of control is large enough, and this controllability time increases when the distance between the initial and final steady states increases. Second, for the evolution case, we use an argument of stabilization in $L^2$ norms in a suitable time to guarantee a local controllability result and so conclude the controllability under positive constraints in the control.

Once the control property has been achieved, this is, we now have nonnegative controls, the classical comparison or maximum principle for parabolic equations concludes the proof of the existence of a positive minimal controllability time.

All previous techniques and results require the control time to be large enough. So, it is natural to analyse whether constrained controllability can be achieved in an arbitrary small time. In \cite{LOH} and \cite{PZ} (for the linear and semilinear heat equations respectively) it was shown  that constrained controllability does not hold when the time horizon is too short. Actually, for quasilinear equations we have the same result.

\section{Statement of the main results}

Let $\Omega \subset \mathbb{R}^N (N \geq 1$ is an integer$)$ be a non-empty bounded connected open set, with regular boundary $\partial \Omega$. Let us fix $T>0$ and let us denote $Q:=\Omega \times (0,T)$ and $\Sigma :=\partial\Omega \times (0,T)$.

	Let\ $\omega, \omega_1 \subset \Omega$ be non-empty open sets, such that $\overline{\omega}_1 \subset \omega$. We deal with the exact controllability to trajectories for the quasilinear system
	\begin{equation}\label{system1}
		\left\{
			\begin{array}{lllll}
			\displaystyle	y_{t} - \nabla \cdot (a(y) \nabla y) = v \varrho_{\omega}  & \text{in} & Q, \\
				 y (x,t)=0             & \text{on} & \Sigma,\\
				  y(x,0)=y_{0}(x)   & \text{in} & \Omega,
			\end{array}
		\right.
	\end{equation}
where $y$ is the associated state, $v$ is the control   and $\varrho_\omega \in C^\infty_0(\overline{\Omega})$, such that $\varrho_\omega \geq 0$\ in $\Omega$,  $\varrho_\omega = 0$\  in $\Omega \backslash \omega$\ \ and\ \ $\varrho_\omega=1$ in $\omega_1$.

	Here, it will be assumed that the real-valued function $a=a(r)$ satisfies
	\begin{equation}\label{a_hypothesis}
	a \in C^{2}(\mathbb{R}),\ \ 0< a_0 \leq a(r) \ \ \text{and}\ \ |a'(r)| +|a''(r)| \leq M,\  \ \forall r \in \mathbb{R}.
	\end{equation}	
	
	We need to introduce some notation, for any $k, \mathit{l} \in \mathbb{N}$, denote by $C^{k,l}(\overline{Q})$ the set of all functions which have continuous derivatives up to order $k$ with respect to the space variable and up to order $l$ with respect to the time variable. For any $\theta \in (0,1)$, put
\begin{align*}
&C^{k+\theta,l+\frac{\theta}{2}}(\overline{Q}) \\
=& \left\{ z \in C^{k,l}(\overline{Q}); \underset{|\sigma|=k}{\sup}\ \underset{ (x_1,t_1) \neq (x_2,t_2)} {\sup} \frac{|\partial^{\sigma}_x \partial^{l}_t z(x_1,t_1)  - \partial^{\sigma}_x \partial^{l}_t z(x_2,t_2) |}{(|x_1-x_2|+|t_1-t_2|^{1/2})^{\theta}} < \infty \right\}
\end{align*}	
and
$$C^{k+\theta}(\overline{\Omega}) = \left\{z \in C^k(\overline{\Omega});  \underset{|\sigma|=k}{\sup}\ \ \ \underset{x_1 \neq x_2}{\sup} \frac{|\partial^{\sigma}_x z(x_1) - \partial^{\sigma}_x z(x_2) |}{|x_1-x_2|^{\theta}} < \infty  \right\},$$	
both of which are Banach spaces with canonical norms.

Note that, if $y_0 \in {C}^{2+1/2}(\overline{\Omega})$, $v \in C^{1/2,1/4}(\overline{Q})$,  then \ref{system1} possesses exactly one solution satisfying
	$$y \in C^{2+1/2,1+1/4}(\overline{Q}),$$
(see for instance~\cite{LADY}, Chapter 5, Section 6, Theorems 6.1 and 6.2). 	
	
We distinguish the following two results: steady state controllability and controllability to target trajectories.
	
\subsection{State state controllability}	
	
\begin{definition}\label{def-st}
	Let  $\overline{v} \in C^{1/2}(\overline{\Omega})$, a function  $\overline{y} \in C^{2+1/2}(\overline{\Omega})$  is said to be a steady state for (\ref{system1}) if it is a solution to
	\begin{equation}\label{steady-state}
-\nabla \cdot (a(\overline{y}) \nabla \overline{y})=\overline{v} \varrho_\omega \;\; \mbox{in}\;\; \Omega,\ \ \ \
\overline{y} =0 \;\; \mbox{in}\;\; \partial\Omega.
	\end{equation}
	The function  $\overline{v} \in C^{1/2}(\overline{\Omega})$ is called the steady control.
\end{definition}

The existence of steady states solution with non-homogeneous values can be analysed using for the nonlinear system the fixed point methods (see \cite{EZ}, Theorem 9.B) and for the linear system the classical existence results (see \cite{LADY2}, Chapter 4, Section 6, Theorem 6.4).

\begin{remark}
\label{obs1}
The application $\Lambda: \overline{v} \mapsto \overline{y}$ shown in \ref{steady-state} is continuous, since $a(\cdot)$ satisfies \ref{a_hypothesis}.

We will denote by\ \ $\mathcal{S}:= \Lambda(C^{1/2}(\overline{\Omega}))$ the set of all the steady-states with steady controls in $C^{1/2}(\overline{\Omega})$.
\end{remark}

\begin{definition}
Fixed $y_0, y_1 \in \mathcal{S}$ and fixed $\overline{v}^0, \overline{v}^1$ such that $\Lambda(\overline{v}^0)=y_0$ and $\Lambda(\overline{v}^1)=y_1$,  we define a  path-connected steady states that drive $y_0$ to $y_1$ as a  continuous path
\begin{align*}
&\gamma : [0,1]  \stackrel{\lambda}{\longrightarrow} C^{1/2}(\Omega) \stackrel{\Lambda}{\longrightarrow} \ \mathcal{S}\\
&\phantom{\gamma : [0}s\ \ \ \longmapsto\ \  \lambda(s)\ \  \longmapsto \gamma(s)=\Lambda(\lambda(s)) ,
\end{align*}
where $\lambda(s)$ is a continuous path of steady controls that drive $\overline{v}^0$ to $\overline{v}^1$ $(\lambda(0)=\overline{v}^0$ and $\lambda(1)=\overline{v}^1)$.

For each $s \in [0,1]$, we denote $\overline{y}^s:=\gamma(s)$ the steady state and $\overline{v}^s:=\lambda(s)$ the steady control of continuous path $\gamma$ .
\end{definition}

\begin{remark}
The existence of a path-connected steady states $\Lambda(s)$ depends of the existence of a continuous path of steady controls $\lambda(s)$ and this result is guaranteed because we have for instance the continuous path $\lambda(s):=(1-s)\overline{v}^0 + s \overline{v}^1$. Actually, there exist an infinite number of path-connected steady states that drive $y_0$ to $y_1$, since there exist an infinite number of continuous path of steady controls that drive $\overline{v}^0$ to $\overline{v}^1$.
\end{remark}

We introduce the first main result for steady states
\begin{theorem}\label{stst-control}
Let $y_0, y_1 \in \mathcal{S}$ fixed and let $\gamma(s):=\overline{y}^s$ be path-connected steady states that drive $y_0$ to $y_1$ with steady control $\overline{v}^s$. Let us assume there exists a constant $\eta>0$ such that
\begin{equation}\label{posit-constraint}
\overline{v}^s \geq \eta, \ \ \ \forall s \in [0,1].
\end{equation}
Then there exists $T_0 >0$ such that, for every $T \geq T_0$ there exists a control $v \in L^{\infty}({\Omega} \times (0,T))$ such that, the system $\ref{system1}$ admits a unique solution $y$ satisfying $y(\cdot, T)=y_1(\cdot)$ in $\Omega$ and $v \geq 0$ in ${\Omega} \times (0,T)$.
\end{theorem}

\subsection{Controllability to target trajectories}

Now, let us define  a target trajectory  $\overline{y}=\overline{y}(x,t)$ as solution to
\begin{equation}\label{trajectory1}
	\left\{
		\begin{array}{lllll}
			\displaystyle	\overline{y}_{t}- \nabla \cdot (a (\overline{y}) \nabla \overline{y}) = \overline{v} \varrho_\omega  & \mbox{in} & Q,\\
			\overline{y} (x,t)=0   & \mbox{on} & \Sigma,\\
			\overline{y}(x,0)= \overline{y}_{0}(x)   & \mbox{in} & \Omega,
		\end{array}
	\right.
\end{equation}
with\ \ $\overline{y}_{0}\in C^{2+1/2}(\overline{\Omega})$ and\ \ $\overline{v} \in C^{1/2,1/4}(\overline{Q})$ such that
\begin{equation}\label{traject1-condition}
 M_a \|\nabla \overline{y}\|_{L^{\infty}(\Omega \times (0,T))} \leq \frac{a_0}{2\ C(\Omega)},
\end{equation}
where $C(\Omega)$ is the Poincar\'e inequality constant, so $\|u\|_{L^2(\Omega)} \leq C(\Omega) \|\nabla u\|_{L^2(\Omega)}$ and the constant $M_a$ is defined by $M_a:=\underset{r \in \mathbb{R}}{\text{sup}}\ |a'(r)|$.

\begin{remark}\label{rem2}
It is clear that we implicitly define the target trajectory $\overline{y}$ that satisfies \ref{trajectory1} and \ref{traject1-condition} in relation to the initial conditions $\overline{y}_0$ and $\overline{v}$ and the function $a(\cdot)$.
We show some observations
\begin{itemize}
	\item[i)] The condition \ref{traject1-condition} must be valid for any $T>0$, in other words, the estimate on the gradient of the target trajectory must be independent of the time variable.
	\item[ii)] If the systems \ref{system1} and \ref{trajectory1} were linear, thus, if the function $a(\cdot)=\text{constant}$, then the condition~\ref{traject1-condition} would be satisfied because in this case we have $M_a = 0$.
	\item[iii)] If the target trajectory was stationary, the condition \ref{traject1-condition} would be satisfied, but then we are in a particular case of Theorem \ref{stst-control}.
	\item[iv)] The proof of the existence of such target trajectory is not obvious. Indeed, if we think in the simple case $(\nabla \overline{y} =0)$ we would have a contradiction. If consider a solution $\overline{y}$ as a function of separable variables, it is difficult to guarantee the existence of such solution.
\end{itemize}
\end{remark}

Due to Remark \ref{rem2}, item iv), we formulate the following second main result for target trajectories:
\begin{theorem}\label{traject-control}
Suppose there exists a target trajectory $\overline{y}$ satisfying the condition \ref{traject1-condition} with initial datum $\overline{y}_0$ and control $\overline{v}$. Let us assume there exist a constant $\eta>0$ such that
\begin{equation}\label{posit-condit}
\overline{v} \geq \eta \ \ \ \text{in}\ \  {\Omega} \times \mathbb{R}^+.
\end{equation}
For any $y_0 \in C^{2+1/2}(\overline{\Omega})$ initial datum, there exists $T_0 >0$ such that for every $T \geq T_0$, we can find a control $v \in L^{\infty}({\Omega} \times (0,T))$ such that the unique solution $y$ to \ref{system1} satisfies \ $y(T)=\overline{y}(T)$ in $\Omega$\ \ and \ $v \geq 0$\ \ in\ \ ${\Omega} \times (0,T)$.

Furthermore, if $y_0 \neq \overline{y}_0$ then the minimal controllability time $T_{\mathrm{min}}$ is strictly positive, where
\begin{equation}\label{mintime}
T_{\mathrm{min}} :=\ \inf \ \{ T>0; \exists\  v \in L^\infty(\Omega \times (0,T))^+,\ \text{such that}\ \ y(T) = \overline{y}(T)\ \text{in}\ \Omega \}.
\end{equation}
\end{theorem}

\begin{remark}
Notice that considering a steady state as a target trajectory, we can think that Theorem~\ref{traject-control} implies Theorem \ref{stst-control}, but this is not the case, since we do not need condition \ref{traject1-condition} for the steady state $($see Remark \ref{rem2}, item iii)).
\end{remark}

The paper is organized as follows.

Section 3 is devoted to prove Theorem \ref{stst-control} using the local controllability result and stair-case method. In~Section 4, we prove the first part of Theorem~\ref{traject-control} using the dissipative property and the local controllability result. In Section 5, we prove the last part of Theorem \ref{traject-control}, this is, we prove the positivity of the minimal controllability time using methods as comparison principle. Section 6 deals with some additional comments and open questions. In Appendix, we will prove the local exact controllability to trajectories for system \ref{system1}.

\section{Proof of Theorem \ref{stst-control}}\label{Sec2}

In this Section, we deal with steady states, for this purpose we use two important results:
\begin{itemize}
	\item[a)] Local exact controllability to trajectories with controls $C^{1/2,1/4}$ (see Appendix).
	\item[b)] The stair-case method to obtain the desired global control (see Figure \ref{Fig1}).
\end{itemize}

\begin{figure}[htp]
\begin{center}
\includegraphics[width=3in]{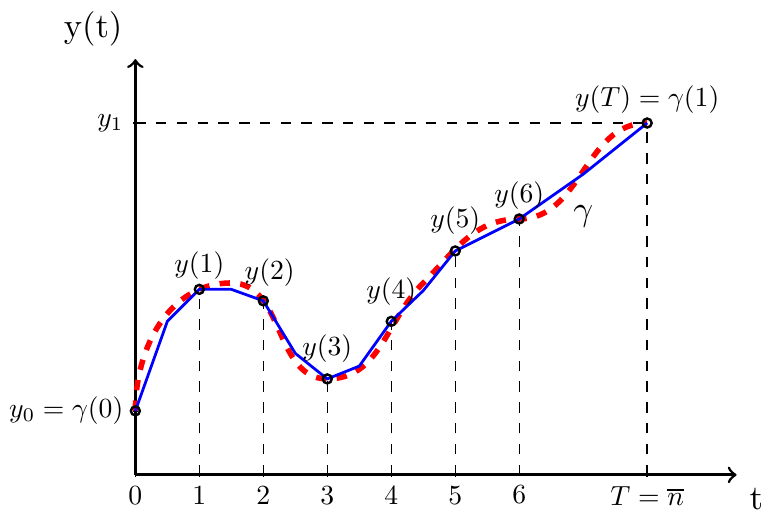}\\
\caption{The stair-case method. {Trajectory of state-control} in blue. {Path of steady states} in red.}\label{Fig1}
\end{center}

\end{figure}

Let $y_0, y_1 \in \mathcal{S}$ fixed and let $\gamma(s):=\overline{y}^s$ be path connected steady states that drive $y_0$ and $y_1$ with steady control $\overline{v}^s$.

\vspace*{4pt}\noindent\textbf{Step 1.} Consequences of Local Controllability. If we suppose that $\|y_1 - y_0\|_{C^{2+1/2}(\overline{\Omega})}$ is small enough, then:

Taking $T=1$, $R >0$ fixed and for any $\epsilon>0$, applying Lemma \ref{lem1} (see Appendix) there exist positive constants $C(R)$ and $\delta_{\epsilon}(R)$ such that if
\begin{equation}
\label{small-init-cond}
\|y_0\|_{C^{2+1/2}(\overline{\Omega})} \leq R, \ \ \|y_1\|_{C^{2+1/2}(\overline{\Omega})} \leq R,\ \ \|y_1 - y_0\|_{C^{2+1/2}(\overline{\Omega})} \leq \delta_\epsilon,
\end{equation}
then, there exists a state-control $(y,v)$ such that
\begin{equation}\label{regu-linear}
v \in C^{1/2,1/4}(\overline{\Omega} \times [0,1]),
\end{equation}
with $y(x,1)=y_{1}(x)$ in  $\Omega$\ \ and we have the following estimate for the control
\begin{equation}\label{estimative-linear}
\|v - \overline{v}^1\|_{C^{1/2,1/4}(\overline{\Omega} \times [0,1]))} \leq C(R) \|y_1 - y_0\|_{C^{2+1/2}(\overline{\Omega})},
\end{equation}
where $\overline{v}^1:=\overline{v}^s|_{s=1}$.

Taking $\delta_{\epsilon}>0$ small enough $($for instance\ \ $\delta_{\epsilon} \leq \dfrac{\epsilon}{C(R)})$, we have
\begin{equation}
\label{estima-control}
\|v - \overline{v}^1\|_{C^{1/2,1/4}(\overline{\Omega} \times [0,1])} \leq \epsilon.
\end{equation}

\vspace*{4pt}\noindent\textbf{Step 2.} Stair-case Method.
If $\|y_1 - y_0\|_{C^{2+1/2}(\overline{\Omega})}$ is not necessarily small, we can work of the following way:

For $\overline{n} \in \mathbb{N}$ let us divide the interval $[0,1]$ in $\overline{n}$ equal parts, and denote
$$\overline{y}_k:=\gamma\left(\frac{k}{\overline{n}}\right),\ k=0,1,2,\dots,\overline{n}-1,\overline{n},$$
where $\overline{y}_{0} = \gamma(0) = y_0$ and $\overline{y}_{\overline{n}}=\gamma(1)=y_1$.

It is clear that $(\overline{y}_k)_{1 \leq k \leq \overline{n}}$ is a finite sequence of steady states and let us denote as $\overline{v}_k:=\overline{v}^{\frac{k}{n}}$ the steady control of $\overline{y}_k$, then by hypothesis \ref{posit-constraint}, we have $\overline{v}_k \geq \eta >0$.

Due to Step 1, it is sufficient to check condition \ref{small-init-cond} to $\overline{y}_{k}$, indeed,  taking 
$R=\underset{s \in [0,1]}{\text{sup}} \|\gamma(s)\|_{C^{2+1/2}(\overline{\Omega})},$ it is obvious that
$$\|\overline{y}_k\|_{C^{2+1/2}(\overline{\Omega})} \leq R, \ \ \  k=0,1,2,\dots,\overline{n}-1,\overline{n}.$$
Taking $\epsilon=\eta$, choosing $\overline{n}$ large enough and using the continuity of path-connected $\gamma(\cdot)$, we get
$$\|\overline{y}_k-\overline{y}_{k-1}\|_{C^{2+1/2}(\overline{\Omega})} \leq \delta_{\eta},$$
where $\delta_{\eta}>0$ is small enough.

Then, for any $1 \leq k \leq \overline{n}$, we can find controls $v_k \in C^{1/2,1/4}(\overline{\Omega} \times~[0,1])$ joining the steady states $\overline{y}_{k-1}$ and $\overline{y}_k$, such that
$$\|v_k-\overline{v}_k\|_{C^{1/2,1/4}(\overline{\Omega} \times [0,1])} \leq \eta.$$
Since $v_k=v_k-\overline{v}_k+\overline{v}_k$, we have
\begin{equation}
\label{positive-aprox-control}
v_k \geq -|v_k-\overline{v}_k| + \overline{v}_k \geq -\|v_k-\overline{v}_k\|_{C^{1/2,1/4}(\overline{\Omega} \times [0,1])} + \overline{v}_k \geq -\eta + \eta =0,\ \ \text{a.e. in}\ \ {\Omega} \times (0,1).
\end{equation}

\vspace*{4pt}\noindent\textbf{Step 3.} Construction of the global control.
For $T = \overline{n}$, we have defined the sequence of connected-paths that runs through the following points:
$$y_0=y(0)\longrightarrow y(1) \longrightarrow \cdots \longrightarrow y(\overline{n}-1) \longrightarrow y(\overline{n})=y_1.$$
For this reason, we define $v: (0, T) \to L^{\infty}({\Omega})$ as
			$$v(t):= v_k(t-(k-1)),\  \ t \in (k-1,k)\ \ \text{for}\ \ k=1, 2, \dots, \overline{n}-1, \overline{n}.$$  	
Thus, we obtain that $v \in L^{\infty}(\Omega \times (0,T))$ is a desired control.

\begin{remark}
Notice that by construction the time employed to control the steady-state is the number of steps we did. Furthermore, due to the particularity of the construction, we can not affirm or deny anything about what happens with the control in a short time (this question will be considered in Section 5).
\end{remark}


\section{Proof of Theorem \ref{traject-control}} \label{Sec3}

In this section, we will prove the controllability to target trajectories, for this purpose we use two important results (see Figure \ref{Fig2}):
\begin{itemize}
\item[a)] Stabilization property.
\item[b)] Local exact controllability to trajectories with control $C^{1/2,1/4}$ (see Appendix).
\end{itemize}

\begin{figure}[htp]
	\begin{center}
  \includegraphics[width=3in]{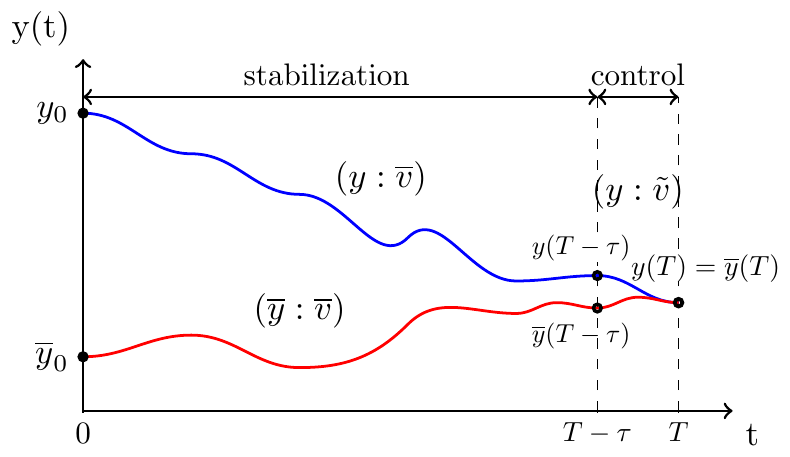}\\
	\caption{Trajectory of state-control in blue. Target trajectory in red.}\label{Fig2}
	\end{center}

\end{figure}

Let $\overline{y}$ be a fixed target trajectory solution to \ref{trajectory1} with control $\overline{v} \in C^{1/2,1/4}(\overline{Q})$ and initial datum $\overline{y}_0 \in C^{2+1/2}(\overline{\Omega})$.

\vspace*{4pt}\noindent\textbf{Step 1.} Stabilization of the system $\bf{(y-\overline{y})(x,t)}$. Let $\tau>0$ be fixed and $T>\tau$ be large enough. In the time interval $[0,T-\tau]$ we control $y$ by means $v=\overline{v}$.

We have the stabilization property of $L^2$ in $[0,T-\tau]$, this is
	\begin{equation}\label{stabpropL2}
		\|y(t)-\overline{y}(t)\|_{L^2(\Omega)} \leq e^{- \lambda t} \|y_0 - \overline{y}_0\|_{L^2(\Omega)},\ \ \forall t \in [0,T-\tau],
	\end{equation}
	for some $\lambda>0$ that does not depend on $T$.
	
	Indeed, by subtracting \ref{trajectory1} to \ref{system1} and multiplying by $(y-\overline{y})$ we have
\begin{align*}
	\int_\Omega (y-\overline{y})_t (y-\overline{y}) dx &- \int_\Omega \nabla \cdot (a(y) \nabla(y-\overline{y})) (y-\overline{y}) dx\\
	& - \int_\Omega \nabla \cdot ((a(y)-a(\overline{y})) \nabla \overline{y}) (y-\overline{y}) dx = 0.
	\end{align*}
	Then
	\begin{align*}
	&\frac{1}{2}  \frac{d}{dt} \left( \|y-\overline{y}\|_{L^2 (\Omega)}^2 \right)  + \int_\Omega a(y) |\nabla (y-\overline{y})|^2 dx \\
	+& \int_\Omega (a(y)-a(\overline{y})) \nabla \overline{y} \cdot \nabla(y-\overline{y}) dx =0.
\end{align*}
Furthermore, using hypothesis about $a(\cdot)$ in \ref{a_hypothesis}, we have
\begin{align*}
	&\left| \int_\Omega (a(y)-a(\overline{y})) \nabla \overline{y} \cdot \nabla(y-\overline{y}) dx \right|\\
	\leq& M_a \|\nabla \overline{y}\|_{L^\infty(\Omega \times (0,T-\tau))}  \int_\Omega |y-\overline{y}| |\nabla (y-\overline{y})| dx\\
	\leq& M_a \|\nabla \overline{y}\|_{L^\infty(\Omega \times (0,T-\tau))} \|y-\overline{y}\|_{L^2(\Omega)} \|\nabla (y-\overline{y})\|_{L^2(\Omega)}\\
	\leq& M_a C(\Omega)\|\nabla \overline{y}\|_{L^\infty(\Omega \times (0,T-\tau))} \|\nabla (y-\overline{y})\|_{L^2(\Omega)}^2,
\end{align*}
where the constants $C(\Omega)$ and $M_a$ are defined in \ref{traject1-condition}.

Finally, joining the previous results and thanks to condition \ref{traject1-condition},  we get
\begin{align*}
	\frac{1}{2}\frac{d}{dt} \left( \|y-\overline{y}\|_{L^2(\Omega)}^2 \right) + a_0 \|\nabla (y-\overline{y})\|_{L^2(\Omega)}^2 \leq \frac{a_0}{2} \|\nabla (y-\overline{y})\|_{L^2(\Omega)}^2,
\end{align*}
then	
\begin{align*}
	\frac{d}{dt} \left( \|y-\overline{y}\|_{L^2(\Omega)}^2 \right) + \frac{a_0}{C(\Omega)} \|y-\overline{y}\|_{L^2(\Omega)}^2 \leq 0,
\end{align*}
where the constant $C(\Omega)$ is defined in \ref{traject1-condition}.

Integrating in time from $0$ to $t$, with $t \in [0,T-\tau]$, we prove \ref{stabpropL2}.

Now, using the stabilization property of $L^2$ in \ref{stabpropL2} with $t=T-\tau$, we finally get
\begin{equation}\label{stab1}
\|y(T-\tau)-\overline{y}(T-\tau)\|_{L^2(\Omega)} \leq  e^{-\lambda(T-\tau)} \|y_0 - \overline{y}_0\|_{L^2(\Omega)}.
\end{equation}

\vspace*{4pt}\noindent\textbf{Step 2.} Local control for the system $\bf{y}(x,t)$. Now, we construct the local control in final time $t=T$.

We can consider ${y}(\cdot,T-\tau)$ as the new initial datum and $\overline{y}|_{\Omega \times (T-\tau,T)}$ as the new target trajectory.

Notice that, for $\eta>0$ (see condition \ref{posit-constraint}) and fixed $T>0$ large enough, we have that $$\|y(T-\tau)-\overline{y}(T-\tau)\|_{L^2(\Omega)} < \delta_{\eta},$$
where $\delta_{\eta}>0$ is small enough.

Taking $R=\|\overline{y}\|_{C^{2+1/2,1+1/4}(\overline{\Omega} \times [T-\tau,T])} +1$, we guaranteed
the hypotheses of Lemma \ref{lem1} (see Appendix), then  there exists a control $\tilde{v} \in C^{1/2,1/4}(\overline{\Omega} \times [T-\tau,T])$, such that $y(T)=\overline{y}(T)$ in $\Omega$.

Furthermore, there exists a positive constant $C(\tau)$ that depend on $\tau$ and does not depend on $T$ such that
$$\|\tilde{v}-\overline{v}\|_{C^{1/2,1/4}(\overline{\Omega}\times [T-\tau,T] )} \leq C(\tau) \|y(T-\tau)-\overline{y}(T-\tau)\|_{L^2(\Omega)} \leq C(\tau) \delta_{\eta}.$$

As $\delta_\eta>0$ is small enough, we get\ \ $\|\tilde{v}-\overline{v}\|_{C^{1/2,1/4}} \leq \eta$.

We split $\tilde{v}=(\tilde{v}-\overline{v})+\overline{v}$ and we conclude
\begin{align*}
\tilde{v}  \geq -|\tilde{v} - \overline{v}| + \overline{v} \geq -\|\tilde{v}-\overline{v}\|_{C^{1/2,1/4}} + \overline{v} \geq -\eta + \eta=0,\ \ \text{in}\ \Omega \times (T-\tau,T).
\end{align*}

\vspace*{4pt}\noindent\textbf{Step 3.} Construction of the global control.
Finally, it is natural to define a required control as
$$v := \begin{cases}
			\overline{v}  &\text{in}\ \ (0,T-\tau),\\
			\tilde{v}	  &\text{in}\ \ (T-\tau,T),
		\end{cases}$$
and this concludes the proof.


\section{Positivity of the minimal controllability time}\label{Sec4}
Let us consider the state-control $(y,v)$ solution to \ref{system1} with initial datum $y_0 \in C^{2+1/2}(\overline{\Omega})$; and let us consider the target  trajectory $\overline{y}$ solution to \ref{trajectory1} with control $\overline{v} \in C^{2+1/2,1/4}(\overline{Q})$, such that\ $\overline{v}\geq \eta>0$ (as in \ref{posit-condit}) and initial datum $\overline{y}_0 \in C^{2+1/2}(\overline{\Omega})$.

Recall the definition in \ref{mintime}:
\begin{equation*}
T_{\mathrm{min}} :=\ \inf\ \{ T>0; \exists\  v \in L^\infty(\Omega \times (0,T))^+,\ \text{such that}\ \ y(T) = \overline{y}(T)\ \text{in}\ \Omega \}.
\end{equation*}
We formulate the following result
\begin{theorem}
We suppose that $y_0 \neq \overline{y}_0$. Then $T_{\mathrm{min}}>0$.
\end{theorem}
\begin{proof}
\textbf{Case 1.} If $y_0 \not < \overline{y}_0$.

By assumptions $y_0 > \overline{y}_0$ in a set of positive measure. Then, there exists a nonnegative $\varphi \in H_0^1(\Omega) \backslash  \{0\}$, such that
$$\int_\Omega  (y_0 - \overline{y}_0) \varphi\ dx >~0.$$
Let us denote $z$ as the solution to \ref{system1} with initial datum $y_0$ and null control. Since\ $z - \overline{y} \in C^0([0,T];H^{-1}({\Omega}))$ and $\Big\langle z(\cdot,0) - \overline{y}(\cdot,0),\varphi(\cdot) \Big\rangle > 0$, we conclude that
\begin{equation}\label{zT-condition}
\Big\langle z(\cdot,T) - \overline{y}(\cdot,T), \varphi(\cdot) \Big\rangle > 0,\ \forall\ T \in [0,T_0),
\end{equation}
with $T_0>0$ small enough.

We will show that\ $T_{\mathrm{min}} \geq T_0$. Indeed, let\ \ $T \in (0,T_0)$ and $v \in L^\infty(\Omega \times (0,T))$ be a nonnegative control such that \ref{system1} admits a solution $y$ with initial datum $y_0$ and control $v$. Then, by the comparison principle (see \cite{MURRAY}, Chapter 3, Section 7, Theorem 12), we have $y \geq z$. Joining this result with \ref{zT-condition}, we have
$$\Big\langle y(\cdot,T),\varphi(\cdot)  \Big\rangle \geq \Big\langle z(\cdot,T),\varphi(\cdot) \Big\rangle > \Big\langle \overline{y}(\cdot,T),\varphi(\cdot) \Big\rangle.$$
Hence $y(\cdot,T) \neq \overline{y}(\cdot,T)$.

\vspace*{4pt}\noindent\textbf{Case 2.} If $y_0 < \overline{y}_0$.

We take $z$ the unique solution to \ref{system1} with initial datum $y_0$ and null control. Then $\xi := y-z$ solves
\begin{equation}\label{xi-equ}
	\left\{
	\begin{array}{lllll}
		\xi_{t}- \Delta (\Phi(\xi + z) -  \Phi(z)) = v \varrho_{\omega} & \mbox{in} & Q, \\
		\xi(x,t)=0        & \mbox{on} & \Sigma,\\
		\xi(x,0)=0   & \mbox{in} & \Omega,
			\end{array}
	\right.	
\end{equation}
where $\displaystyle \Phi(r):= \int_0^r a(s) ds$.

Besides, $\overline{\xi}:=\overline{y}-z$ solves \ref{xi-equ} with initial datum $\overline{y}_0-y_0$ and control $\overline{v}$. The problem is reduced to prove  the existence of $T_0>0$ such that, for any $T~\in~(0,T_0)$ and for any nonnegative $v \in L^\infty(\omega_1 \times (0,T))$, we obtain $\xi(\cdot,T)~\neq~\overline{\xi}(\cdot,T)$. Clearly, this implies $y(\cdot,T) \neq \overline{y}(\cdot,T)$.

Suppose by contradiction, for any $T_0>0$ there exists $T \in (0,T_0)$ such that $\xi(\cdot,T) = \overline{\xi}(\cdot,T)$.

$\xi$ is characterized by the duality identity
\begin{equation}\label{duality}
\Big\langle \xi(\cdot,T),\varphi^T(\cdot) \Big\rangle = \iint_{\omega \times (0,T)} v \varrho_\omega \varphi\  dxdt,
\end{equation}
where $\varphi$ is the solution to the adjoint problem
\begin{equation}\label{varphi-equ}
	\left\{
	\begin{array}{lllll}
		-\varphi_{t}- \Psi_\xi(z) \Delta \varphi = 0 & \mbox{in} & Q, \\
		\varphi(x,t)=0        & \mbox{on} & \Sigma,\\
		\varphi(x,T)=\varphi^T(x)   & \mbox{in} & \Omega,
			\end{array}
	\right.	
\end{equation}
with $\Psi_\xi$ satisfying
$$\Psi_\xi(s)= \begin{cases}
						\dfrac{\Phi(\xi + s) - \Phi(\xi)}{s} 	&\text{if}\  \ s \neq 0,\\
						\Phi'(\xi) 	&\text{if}\ \ s=0.
					\end{cases}$$

Let $\phi_1$ be the first eigenfunction of the Dirichlet Laplacian in $\Omega$, which is strictly positive in $\Omega$ (see \cite{Evans}, Chapter 6, Section 5, Theorem 3). For any $r>0$, define the set
$$E_r := \{ x \in \Omega \backslash \omega\ |\ \mathrm{dist}(x,\partial \omega) < r \}.$$
We consider a constant $\theta>0$ such that
$$\int_{\Omega \backslash (\omega \cup E_d)} (-\phi_1)(\overline{y}_0 - y_0) dx \leq -\theta <0,$$
where $d:=\mathrm{dist}(\partial \omega, \partial \Omega)/2$, then we define
$$C_\theta:= \frac{\theta}{3 \|\phi_1\|_{L^\infty(\Omega)} \|\overline{y}_0 - y_0\|_{L^1(\Omega)}}>0.$$
Let us consider the cut-off function $\zeta \in C^\infty(\overline{\Omega})$ such that
\begin{equation*}
	\zeta(x) := \begin{cases}
				 -1		&\text{if}\ \ x \in \Omega \backslash (\omega \cup E_\delta),\\
				 -1 \leq \zeta(x) \leq C_\theta   &\text{if}\ \ x \in E_\delta,\\	
	 			 {C_\theta} &\text{if}\ \ x \in \omega,
			 \end{cases}
\end{equation*}
for some $\delta>0$.

We define $\varphi^T(x) := \zeta(x) \phi_1(x)$ (see Figure \ref{Fig3}), then  there exists a constant $\tilde{\theta}>0$ such that
\begin{equation*}
	\varphi^T(x) \geq \tilde{\theta},\  \forall  x \in \omega.
\end{equation*}


\begin{figure}[htp]
	\begin{center}
	\includegraphics[width=3in]{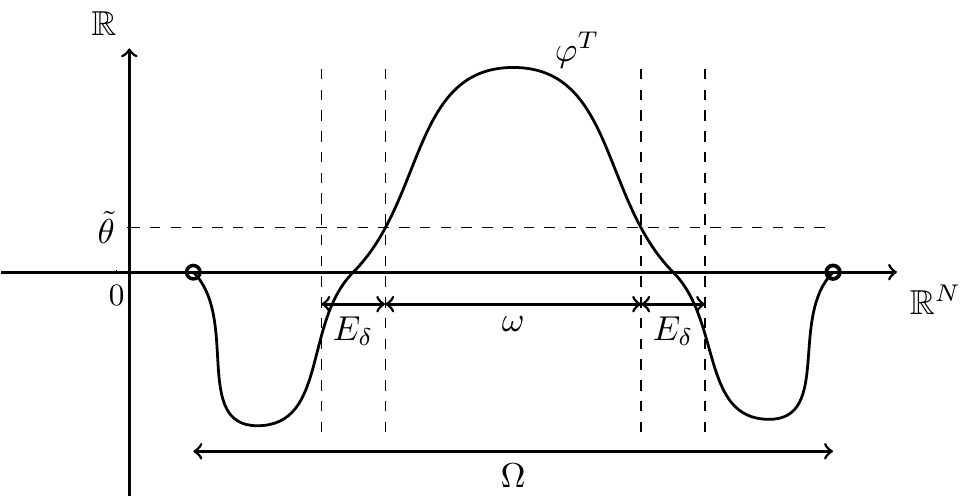}\\
\caption{Initial datum $\varphi^T$ for the adjoint system.}\label{Fig3}
 \end{center}
\end{figure}

We will prove that $\Big\langle \overline{\xi}(\cdot,T),\varphi^T(\cdot) \Big\rangle < 0$. Indeed,
\begin{align*}
	\int_\Omega (\overline{y}_0 - y_0) \varphi^T dx =& \int_{\Omega \backslash (\omega \cup E_\delta)} (\overline{y}_0 - y_0) \varphi^T dx + \int_{E_\delta} (\overline{y}_0 - y_0) \varphi^T dx \\
	&+ \int_\omega (\overline{y}_0 - y_0) \varphi^T dx.
\end{align*}
For $\delta>0$ small enough, we get, on the one hand,
\begin{align*}
	\int_{\Omega \backslash (\omega \cup E_\delta)} (\overline{y}_0 - y_0) \varphi^T dx &= \int_{\Omega \backslash (\omega \cup E_\delta)} (\overline{y}_0 - y_0) (-\phi_1) dx\\
	&\leq \int_{\Omega \backslash (\omega \cup E_d)} (\overline{y}_0 - y_0)(- \phi_1) dx \leq -\theta.
\end{align*}
On the other hand,
\begin{equation*}
	\left|\int_{E_\delta} (\overline{y}_0 - y_0) \varphi^T dx \right| \leq C \|\overline{y_0} - y_0\|_{L^\infty(\Omega)} \|\phi_1\|_{L^\infty(\Omega)} |E_\delta| \leq \frac{\theta}{3}.
\end{equation*}
Finally
\begin{equation*}
	\left|\int_\omega (\overline{y}_0 - y_0) \varphi^T dx\right| = \left|\int_\omega (\overline{y}_0 - y_0) C_\theta \phi_1 dx \right| \leq C_\theta \|\phi_1\|_{L^\infty(\Omega)} \|\overline{y}_0 - y_0\|_{L^1(\Omega)} = \frac{\theta}{3}.
\end{equation*}
Then
\begin{equation}\label{ineq1}
	\int_\Omega (\overline{y}_0 - y_0)  \varphi^T dx  \leq -\frac{\theta}{3}< 0.
\end{equation}

By transposition results, $\overline{\xi} \in C^0([0,T];H^{-1}(\Omega))$ (see \cite{Lions2}, Chapter 3, Section 4.7). Hence, choosing $T_0 > 0$ small enough, from \ref{ineq1} we conclude
\begin{equation}\label{xiT-negat}
\Big\langle \overline{\xi}(\cdot,T), \varphi^T(\cdot) \Big\rangle < 0,\ \ \forall\ T \in (0,T_0].
\end{equation}

We will prove that $\varphi_- =0$ in $\omega \times (0,T_0)$. Indeed, we have that $\varphi^T \in W^{2,p}(\Omega) \cap W^{1,p}_0(\Omega)$ (see \cite{Bre}, Chapter 9, Section 9, Theorem 9.32), then by regularity results, we have that $\varphi \in C(\Omega \times [0,T_0])$ (see \cite{Lieb}, Chapter 7, Section 7, Theorem 7.32). So, since $\varphi^T \geq \tilde{\theta}>0$ in $\omega$, by continuity for $T_0>0$ small enough, we get $\varphi(x,t) \geq \tilde{\theta} >0$\ \ in $(x,t) \in \omega \times (0,T_0)$, thus
\begin{equation}\label{var-nonposit}
\varphi_- = 0,\ \ \text{in}\   \omega \times (0,T_0).
\end{equation}

Substituting \ref{xiT-negat} and \ref{var-nonposit} in \ref{duality}, for $T \in (0,T_0)$, we get
\begin{equation*}
0 > \Big\langle {\xi}(\cdot,T),\varphi^T \Big\rangle = \iint_{(0,T) \times \omega} v \varrho_\omega \varphi\ dxdt \geq 0.
\end{equation*}
This is a contradiction, hence $\xi(\cdot,T) \neq \overline{\xi}(\cdot,T)$.
\end{proof}

\begin{remark}
The result of this Section is independent of how we obtain the controllability under constraints for target trajectories, thus, if there exists $T_{\mathrm{min}}$ and the initial datum for systems \ref{system1} and \ref{trajectory1} are different, then $T_{\mathrm{min}}>0$.

\end{remark}

\section{Additional comments and results}\label{Sec5}

\begin{itemize}
	\item It is clear that Theorems \ref{stst-control} and \ref{traject-control} are verified if we add a semilinear term in the system \ref{system1}, this is, if we consider the following system
	 \begin{equation}
 \label{eq7}
\left\{
\begin{array}{lllll}
y_{t} -  \nabla \cdot (a(y)\nabla y) + F(y) = v \varrho_{\omega} &\text{in} &Q,\\
y(x,t)  = 0 &\text{on} &\Sigma,\\
y(x,0) = y_{0}(x) &\text{in} &\Omega,
\end{array}
\right.
\end{equation}
where in this case $F(\cdot)$ satisfies same conditions as in \cite{PZ}.
	\item The spaces $C^{1/2,1/4}(\overline{Q})$ and $C^{2+1/2}(\overline{\Omega})$ in both Theorems \ref{stst-control} and \ref{traject-control} can be replaced by $C^{\theta,\theta/2}(\overline{Q})$ and $C^{2+\theta}(\overline{\Omega})$ for any $\theta \in (0,1)$, respectively. This is clear by the approach developed in this paper.
	\item The condition \ref{traject1-condition} is fundamental in this paper, we need this inequality to prove the control to trajectories and the stabilization property, in other words, the result of Theorem \ref{traject-control} depend on the existence of this condition. It would be very important to prove the same results without this condition, this question doesn't seem easy because we would need new especial estimates to control this problem.
	\item It is interesting to study the constraint controllability for more quasi-linear parabolic equations, for instance, when the nonlinearity $a(y)$ is replaced by the nonlinearity $a(\nabla y)$ in \ref{system1}, it seems that the techniques applied in this paper are not enough. We need more regularity for the coefficient of principal part to prove local controllability results and to prove comparison principle.
\end{itemize}

\begin{itemize}
	
\item \textbf{Controllability for the Coupled System}

If we consider the system
 \begin{equation}
 \label{eq11}
\left\{
\begin{array}{lllll}
y_{1,t} -  \nabla \cdot (a(y_1)\nabla y_1) + F_1(y_1,y_2) = v \varrho_{\omega} &\text{in} &Q,\\
y_{2,t} -  \nabla \cdot (a(y_2)\nabla y_2) + F_2(y_1,y_2) = 0 &\text{in}  &Q,\\
y_1(x,t) = y_2(x,t) = 0 &\text{on} &\Sigma,\\
y_1(x,0) = y_{1,0}(x),\ \ \ y_2(x,0) = y_{2,0}(x) &\text{in} &\Omega,
\end{array}
\right.
\end{equation}
with the target trajectory
\begin{equation}
 \label{eq10}
\left\{
\begin{array}{lllll}
\overline{y}_{1,t} -  \nabla \cdot (a(\overline{y}_1)\nabla \overline{y}_1) + F_1(\overline{y}_1,\overline{y}_2) = \overline{v} \varrho_{\omega} &\text{in} &Q,\\
\overline{y}_{2,t} -  \nabla \cdot (a(\overline{y}_2)\nabla \overline{y}_2) + F_2(\overline{y}_1,\overline{y}_2) = 0 &\text{in} &Q,\\
\overline{y}_1(x,t) = 0,\ \  \overline{y}_2(x,t) = 0&\text{on} &\Sigma,\\
\overline{y}_1(x,0) = \overline{y}_{1,0}(x),\ \  \overline{y}_2(x,0) = \overline{y}_{2,0}(x) &\text{in} &\Omega.
\end{array}
\right.
\end{equation}
We want to find a control $v \in L^\infty(\Omega \times (0,T))$, such that the solution $(y_1,y_2)$ of  \ref{eq11} satisfy $(y_1(\cdot, T), y_2(\cdot, T)) = (\overline{y}_1(\cdot, T),\overline{y}_2(\cdot,T))$ in $\Omega$ and $v \geq 0$\ \ if\ \ $\overline{v}\geq C >0$.

This problem is an open question, the difficulty to apply the comparison principle is fundamental in the method of solution. It will be necessary to implement other tools.


\item \textbf{Degenerate Null Controllability}

In dimension $N=1$, if we consider the system
\begin{equation}
	\label{eq9}
\left\{
\begin{array}{lllll}
y_t - (a(y) x^{\alpha}  y_x)_x + F(t,x,y) = v \varrho_{\omega} &\text{in} &Q,\\
y(1,t) = 0\ \ \text{and} \begin{cases}
y(0,t) = 0\ \ &\text{for}\ 0 \leq \alpha <1,\\
(x^{\alpha} y_x)(0,t) = 0\ \ &\text{for}\ 1\leq \alpha <2,
\end{cases} &\text{on} &\Sigma,\\
y(x,0) = y_{0}(x) &\text{in} &\Omega.
\end{array}
\right.
\end{equation}
We want to find a control $v \in L^2(\omega \times (0,T))$, such that the solution $y$ of \ref{eq9} satisfies\ $y(\cdot,T) = 0$ in $\Omega$.

The linear system will be
\begin{equation}
	\label{eq8}
\left\{
\begin{array}{lllll}
&y_t - (a(0) x^{\alpha}  y_x)_x + c(t,x)y = v \varrho_{\omega} &\text{in} &Q,\\
&y(1,t) = 0\ \ \text{and} \begin{cases}
y(0,t) = 0\ \ &\text{for}\ 0 \leq \alpha <1,\\
(x^{\alpha} y_x)(0,t) = 0\ \ &\text{for}\ 1\leq \alpha <2,
\end{cases} &\text{on} &\Sigma,\\
&y(x,0) = y_{0}(x) &\text{in} &\Omega.
\end{array}
\right.
\end{equation}
In suitable conditions for the function $a(\cdot)$ it is possible to prove the linear problem (for instance see \cite{DNC}). But the non-linear problem is a difficult result to obtain as more tools are needed to understand the spaces that must be used.

\end{itemize}

\section{Appendix}\label{Sec6}

\subsection{Local controllability result}
We will prove a local controllability result, thus:
\begin{lemma}\label{lem1}
Given $T>0$ and $R>0$. Then, there exist positive constants $C$ and $\delta$ depending on $R$ and $T$ such that, for all targets $\overline{y} \in C^{2+1/2,1+1/4}(\overline{Q})$ solution to \ref{trajectory1} with initial datum $\overline{y}_0$ and control $\overline{v}$ and for each initial datum $y_0 \in C^{2+1/2}(\overline{\Omega})$, such that if
	\begin{equation}\label{localcondition}
		\|y_0\|_{C^{2+1/2}(\overline{\Omega})} \leq R,\ \ \|\overline{y}\|_{C^{2+1/2,1+1/4}(\overline{Q})} \leq R\ \ \text{and}\ \  \|y_0 - \overline{y}_0\|_{C^{2+1/2}(\overline{\Omega})} \leq \delta,
	\end{equation}
	we can find a control $v \in C^{1/2,1/4}(\overline{Q})$, such that the corresponding solution $y$ of \ref{system1} satisfies
		$$y(T) = \overline{y}(T).$$
	Moreover,
		$$\|v - \overline{v}\|_{C^{1/2,1/4}(\overline{Q})} \leq C \|y_0 - \overline{y}_0\|_{C^{2+1/2}(\overline{\Omega})}.$$			
\end{lemma}

\begin{proof}
	We will use arguments of \cite{BEC} and \cite{XU} because the linearized problems are very similar.

Taking $z=y-\overline{y}$\ and\ $u=v-\overline{v}$, the problem is reduced to prove the local null controllability of the system
\begin{equation}\label{nonlinearsystem}
	\left\{
	\begin{array}{lllll}
		z_{t}- \nabla \cdot (\alpha_z(x,t) \nabla z) + \nabla \cdot (\beta_z(x,t) \nabla \overline{y}\ z) = u \varrho_{\omega} & \mbox{in} & Q, \\
		z(x,t)=0        & \mbox{on} & \Sigma,\\
		z(x,0)=y_{0}(x) - \overline{y}_0(x)   & \mbox{in} & \Omega,
			\end{array}
	\right.	
\end{equation}
where $$\alpha_{z}(x,t):= a(z(x,t)+\overline{y}(x,t))$$
and
$$\beta_{z}(x,t):= \begin{cases}
				- \dfrac{a(z(x,t)+\overline{y}(x,t))-a(\overline{y}(x,t))}{z(x,t)}&, \  \text{if}\ \ z(x,t)\neq 0, \\
				-a'(\overline{y}(x,t))&, \  \text{if}\ \ z(x,t)=0.
				\end{cases}$$
		
Let us fix $w \in Z:= \{ w \in  C^{1+1/2,1+1/4}(\overline{Q});\ w(\cdot,0)=z_0(\cdot) \}$ and we consider the linear system
\begin{equation}\label{linearsystem}
	\left\{
	\begin{array}{lllll}
		z_{t}- \nabla \cdot (\alpha_w(x,t) \nabla z) + \nabla \cdot (\beta_w(x,t) \nabla \overline{y}\ z) = u \varrho_{\omega} & \mbox{in} & Q, \\
		z(x,t)=0        & \mbox{on} & \Sigma,\\
		z(x,0)=z_{0}(x):=y_{0}(x) - \overline{y}_0(x)   & \mbox{in} & \Omega,
			\end{array}
	\right.	
\end{equation}
where
\begin{equation}
\alpha_w(x,t) \in C^{1,1}(\overline{Q})
\end{equation}
and
\begin{equation}
\beta_w(x,t) \nabla \overline{y} \in C(\overline{Q}).
\end{equation}

Let us consider of adjoint system associated to the linear system \ref{linearsystem}
\begin{equation}\label{adjointsystem}
	\left\{
	\begin{array}{lllll}
		-p_{t}- \nabla \cdot (\alpha_w(x,t) \nabla p) + \nabla \cdot (\beta_w(x,t) \nabla \overline{y}\ p) = 0 & \mbox{in} & Q, \\
		p(x,t)=0        & \mbox{on} & \Sigma,\\
		p(x,T)=p_{T}(x)   & \mbox{in} & \Omega.
			\end{array}
	\right.	
\end{equation}

We denote
\begin{equation}
\label{constantB}
B:=\left(1+ \|w\|_{C^{1,1}(\overline{Q})}^2 + \|\nabla \overline{y}\|_{L^{\infty}({Q})}^2 \right).
\end{equation}

Given $\omega_0$ an open and nonempty subset of $\omega$ such that $\overline{\omega}_0 \subset \omega$, there exists a function $\alpha_0 \in C^{2}(\overline{\Omega})$ such that
$$\alpha_0(x)>0\ \text{in}\ \Omega;\ \alpha_0(x)=0\ \text{on}\ \partial \Omega; |\nabla \alpha_0(x)| >0\ \text{in} \ \overline{\Omega \backslash \omega_0}.$$
For any $\lambda>0$, we define
$$\phi(x,t):=\frac{e^{\lambda \alpha_0(x)}}{t(T-t)},\ \ \ \alpha(x,t):=\frac{e^{\lambda \alpha_0(x)}-e^{2\lambda |\alpha_0|_{C(\overline{\Omega})}}}{t(T-t)}.$$

Proceeding as in \cite{XU}, we can prove the following observability inequality
\begin{equation}\label{obs-ineq}
	\|p(0)\|_{L^2(\Omega)}^2 \leq C e^{e^{CB}} \iint_{\omega_1 \times (0,T)} e^{2s \alpha} \phi^3 |p|^2 dxdt,
\end{equation}
where $C>0$ is a constant, $B$ is the constant given in \ref{constantB}, $\lambda \geq CB$ and $s \geq C e^{CB}$.

Similar to Proposition 4.1 in \cite{XU}, we can prove the null controllability of the linear system \ref{linearsystem}, thus, there exist a control $v \in C^{1/2,1/4}(\overline{Q})$ such that the associated solution $z$ of \ref{linearsystem} satisfies $z(T)=0$ in $\Omega$.
The estimate for the control is
\begin{equation}\label{estimate-control}
\|v-\overline{v}\|_{C^{1/2,1/4}(\overline{Q})} \leq C e^{e^{CB}} \|z_0\|_{L^2(\Omega)}.
\end{equation}

Finally, using arguments of Fixed-Point (specifically  Kakutani Fixed-Point Theorem), we can conclude the proof of Lemma~\ref{lem1}. Indeed, we define the set
$$K:=\{z \in C^{2+1/2,1+1/4}(\overline{Q}); \|z\|_{C^{2+1/2,1+1/4}(\overline{Q})} \leq 1,\ \ z(\cdot,0)=z_0(\cdot)\},$$
and a (possible multivalued) map $\Lambda: Z \rightarrow 2^{Z}$ such that for any $w \in Z$, put
\begin{align*}
&\Lambda(w):=\{ z \in Z;\ \exists\ v \in C^{1/2,1/4}(\overline{Q})\ \ \text{and a constant}\ C>0\ \ \text{such that}\ (z,v)\\
&\phantom{\Lambda(w):=\{ y \in Z; \exists\ v}  \text{satisfy}\ \ref{linearsystem}, \ref{estimate-control}\ \ \text{and}\  \ z(T)=0\ \  \text{in}\ \Omega\}.
\end{align*}
The map $\Lambda$ is well defined. Indeed, the null controllability of the linear system \ref{linearsystem} guaranteed this.

If initial datum $z_0$ is small enough, then $\Lambda(K) \subset K$. Indeed, by Schauder theory of linear parabolic systems we have the estimates for the state $z$:
$$\|z\|_{C^{2+1/2,1+1/4}(\overline{Q})}\leq C(B)(\|z_0\|_{C^{2+1/2}(\overline{\Omega})} + \|v-\overline{v}\|_{C^{1/2,1/4}(\overline{Q})}).$$
Taking $z_0$ small enough, we get\ \ $\|z\|_{C^{2+1/2,1+1/4}(\overline{Q})}\leq 1$.

Furthermore, it is clear that $K$ is a nonempty convex and compact subset of $Z$, this is an immediate result because we know that $K$ is compactly embedded in $Z$.

Also $\Lambda$ has a closed graph in $Z$.

Therefore, if $z_0$ is small enough by Kakutani Fixed-Point Theorem (see \cite{EZ}), $\Lambda$ possesses at least one fixed point $z$. This means that for quasilinear \ref{nonlinearsystem}, there exist a control $v \in C^{1/2,1/4}(\overline{Q})$ such that the solution $z$ satisfy $z(T)=0$ in $\Omega$. The cost of the control function $v(x,t)$ verifies \ref{estimate-control}.

This completes the proof of Lemma \ref{lem1}.
\end{proof}




\section*{Acknowledgments} First, mention to my colleagues Dario Pighin and Borjan Geshkovski from DeustoTech - University of Deusto and Dany Nina from Universidade Federal Fluminense for their collaboration in some suggestions to this article.

Special mention to PhD. Enrique Zuazua, for his advice and support in my stay of the year 2018 in DyCon ERC Advanced Grant Project, DeustoTech - Deusto Foundation, Bilbao, Spain.


\medskip
Received  October 2019; 1st revision May 2020; final revision March 2021.
\medskip


\begin{thebibliography}{99}
\bibitem{ATF} (MR924574) [10.1007/978-1-4615-7551-1]
\newblock V. M. Alekseev, V. M. Tikhomorov and S. V. Formin,
\newblock \emph{Optimal Control},
\newblock Consultants Bureau, New York, 1987.

\bibitem{BEC} (MR2009498) [10.1155/S1085337503303033]
\newblock M. Beceanu,
\newblock \doititle{Local exact controllability of the diffusion equation in one dimension},
\newblock \emph{Abstr. Appl. Anal.}, \textbf{2003} (2003), 793--811.

\bibitem{Bre} (MR2759829)
\newblock H. Brezis,
\newblock \emph{Functional Analysis, Sobolev Spaces and Partial Differential Equations},
\newblock Universitext, Springer New York, 2010. Available from: \url{https://books.google.es/books?id=GAA2XqOIIGoC}.

\bibitem{DNC} (MR3950702) [10.1007/s00028-019-00487-8]
\newblock R. Du,
\newblock \doititle{Null controllability for a class of degenerate parabolic equations with the gradient terms},
\newblock \emph{J. Evol. Equ.}, \textbf{19} (2019), 585--613.

\bibitem{Im} (MR1349016) [10.1070/SM1995v186n06ABEH000047]
\newblock O. Yu. \'Emanuvilov,
\newblock \doititle{Controllability of parabolic equations},
\newblock \emph{Mat. Sb.}, \textbf{186} (1995), 879--900.

\bibitem{Evans} (MR2597943) [10.1090/gsm/019]
\newblock L. C. Evans,
\newblock \emph{Partial Differential Equations},
\newblock Graduate studies in mathematics, American Mathematical Society, 2010. Available from: \url{https://books.google.es/books?id=Xnu0o_EJrCQC}.

\bibitem{Mig} (MR3736163) [10.1007/s10957-017-1190-4]
\newblock E. Fern\'andez-Cara, D. Nina-Huam\'an, M. R. Nu\~nez-Ch\'avez and F. B. Vieira,
\newblock \doititle{On the theoretical and numerical control of a one-dimensional nonlinear parabolic partial differential equation},
\newblock \emph{J. Optim. Theory Appl.s}, \textbf{175} (2017), 652--682.

\bibitem{FZ} (MR1750109)
\newblock E. Fern\'andez-Cara and E. Zuazua,
\newblock The cost of approximate controllability for heat equations: The linear case,
\newblock \emph{Adv. Differential Equations}, \textbf{5} (2000), 465--514.

\bibitem{FI} (MR1406566)
\newblock A. V. Fursikov and O. Yu. Imanuvilov,
\newblock \emph{Controllability of Evolution Equations},
\newblock Lecture Notes Series, vol.~34, Seoul National University, Research Institute of Mathematics, Global Analysis Research Center, Seoul, 1996.


\bibitem{ImYa} (MR1987865) [10.2977/prims/1145476103]
\newblock O. Yu. Imanuvilov and M. Yamamoto,
\newblock \doititle{Carleman inequalities for parabolic equations in Sobolev Spaces of negative order and exact controllability for semilinear parabolic equations},
\newblock \emph{Publ. Res. Inst. Math. Sci.}, \textbf{39} (2003), 227--274.

\bibitem{LADY2} (MR0244627)
\newblock O. A. Ladyzhenskaya and N. N. Ural'ceva,
\newblock \emph{Linear and Quasilinear Elliptic Equations},
\newblock Translations by Scripta Technica, Inc, Academy Press, New York and London, 1968.

\bibitem{LADY} (MR0241822)
\newblock O. A. Lady\v{z}henskaya, V. A. Solonnikov and N. N. Ural'ceva,
\newblock \emph{Linear and Quasilinear Equations of Parabolic Type},
\newblock Translations of Mathematical Monographs, vol.~23, AMS, Providence, RI, 1968. Available from: \url{https://books.google.es/books?id=dolUcRSDPgkC}.

\bibitem{LEBAL} (MR3912679) [10.1016/j.jde.2018.08.046]
\newblock K. Le Balc'h,
\newblock \doititle{Controllability of a $4 \times 4$ quadratic reaction-diffusion system},
\newblock \emph{J. Differential Equations}, \textbf{266} (2019), 3100--3188.

\bibitem{LEB}
\newblock G. Lebeau and L. Robbiano,
\newblock Contr\^{o}le exact de l\'equation de la chaleur,
\newblock \emph{Communications in Partial Differential Equations}, \textbf{20} (1995), 335--356.

\bibitem{Lieb} (MR1465184) [10.1142/3302]
\newblock G. M. Lieberman,
\newblock \emph{Second Order Parabolic Differential Equations},
\newblock World Scientific, 1996. Available from: \url{https://books.google.es/books?id=s9Guiwylm3cC}.

\bibitem{Lions} (MR870385) [10.1007/BFb0007542]
\newblock J.-L. Lions,
\newblock \doititle{Controlablit\'e exacte des syst$\grave{e}$mes distribu\'es: Remarques sur la th\'eorie g\'en\'erale et les applications},
\newblock \emph{Analysis and Optimization System}, Springer, (1986), 3--14.

\bibitem{Lions2} 
\newblock J. L. Lions and E. Magenes,
\newblock \emph{Problemes aux Limites non Homogenes et Applications},
\newblock Grundlehren der mathematischen Wissenschaften, Springer Berlin Heidelberg, vol~1, 1968.

\bibitem{XU} (MR2974729) [10.1137/110851808]
\newblock X. Liu and X. Zhang,
\newblock \doititle{Local controllability of multidimensional quasi-linear parabolic equations},
\newblock \emph{SIAM J. Control Optim.}, \textbf{50} (2012), 2046--2064.

\bibitem{LOH} (MR3669834) [10.1142/S0218202517500270]
\newblock J. Loh\'eac, E. Tr\'elat and E. Zuazua,
\newblock Minimal controllability time for the heat equation under unilateral state or control constraints,
\newblock \emph{Math. Models Methods Appl. Sci.}, \textbf{27} (2017), 1587--1644. 

\bibitem{PZ} (MR3917471) [10.3934/mcrf.2018041]
\newblock D. Pighin and E. Zuazua,
\newblock \doititle{Controllability under positive constraints of semilinear heat equations},
\newblock \emph{Math. Control Relat. Fields}, \textbf{8} (2018), 935--964.

\bibitem{POU} (MR3918081) [10.1088/1361-6544/aaf07e]
\newblock C. Pouchol, E. Tr\'elat and E. Zuazua,
\newblock \doititle{Phase portrait control for 1d monostable and bistable reaction-diffusion equations},
\newblock \emph{Nonlinearity}, \textbf{32} (2019), 884--909.

\bibitem{MURRAY} 
\newblock M.  H. Protter and H. F. Weinberger,
\newblock \emph{Maximum Principles in Differential Equations},
\newblock Springer, New York, 2012. Available from:  \url{https://books.google.es/books?id=JUXhBwAAQBAJ}.

\bibitem{Schm} (MR986155) [10.1016/0022-0396(89)90077-6]
\newblock E. J. P. G. Schmidt,
\newblock \doititle{Boundary control for the heat equation with steady-state targets},
\newblock \emph{J. Differential Equations}, \textbf{78} (1989), 89--121.

\bibitem{EZ} (MR816732) 
\newblock E. Zeidler,
\newblock \emph{Nonlinear Functional Analysis and its Applications I: Fixed-Point Theorems},
\newblock Springer-Verlag, New York, Berlin, Heidelberg, Tokio, 1986.

 

\end{thebibliography}
\end{document}